\documentclass{amsart}%
\usepackage{amsfonts}
\usepackage{amsmath}
\usepackage{amssymb}
\usepackage{graphicx}%
\providecommand{\U}[1]{\protect\rule{.1in}{.1in}}
\newtheorem{theorem}{Theorem}

\newtheorem{corollary}[theorem]{Corollary}

\newtheorem{lemma}[theorem]{Lemma}

\newtheorem{proposition}[theorem]{Proposition}

\begin{document}

\title{A Bernstein result and counterexample for entire solutions to Donaldson's equation}
\author{Micah Warren\\University of Oregon }
\address{Fenton Hall\\
University of Oregon\\ Eugene, OR 97403}
\email{micahw@uoregon.edu}
\thanks{ The author's work is supported in part by the NSF via DMS-1161498}
\maketitle

\begin{abstract}
We show that convex entire solutions to Donaldson's equation are quadratic,
using a result of Weiyong He. \ We also exhibit entire solutions to the
Donaldson equation that are not of the form discussed by He. \ In the process
we discover some non-trivial entire solutions to complex Monge-Amp\`{e}re equations.

\end{abstract}

\section{Introduction}

In this note we show the following. \ 

\begin{theorem}
Suppose that $u$ is a convex solution to the Donaldson equation on $%
\mathbb{R}
\times%
\mathbb{R}
^{n-1}=(t,x_{2},...,x_{n})$
\begin{equation}
\tilde{\sigma}_{2}(D^{2}u)=u_{11}(u_{22}+u_{33}+...+u_{nn})-u_{12}%
^{2}-...-u_{1n}^{2}=1.\label{edon}%
\end{equation}
Then $u$ is a quadratic function. \ 
\end{theorem}

Donaldson introduced the operator%
\[
Q(D^{2}u)=u_{tt}\Delta u-|\nabla u_{t}|^{2}%
\]
arising in the study of the geometry of the space of volume forms on compact
Riemannian manifolds \cite{Don}. \  On Euclidean space, (\ref{edon})  becomes
an interesting non-symmetric fully nonlinear equation. \ Weiyong He has
studied aspects of entire solutions on Euclidean space, and was able to show
that \ \cite[Theorem 2.1]{HePJM} if $u_{11}=const,$ then the solution can be
written in terms of solutions to Laplace equations. \ \ 

Here we show that any convex solution must also satisfy $u_{11}=const.$ \ It
follows quickly that the solution must be quadratic. \ We also show that, in
the absence of the convexity constraint, solutions exists for which
$u_{11}=const$ fails.

\begin{theorem}
There exists solutions to the Donaldson equation which are not of the form
given by He. \ \ \ 
\end{theorem}

In real dimension $3$ we note that solutions of (\ref{edon}) can be extended
to solutions of the  the  complex Monge-Amp\`{e}re equation on $%
\mathbb{C}
^{2}$
\begin{equation}
\det\left(  \partial\bar{\partial}u\right)  =1\label{cma}%
\end{equation}
and we can conclude the following. \ 

\begin{corollary}
There exist a nonflat solution of the complex Monge-Amp\`{e}re equation
(\ref{cma}) on $%
\mathbb{C}
^{2}$ whose potential depends on only three real variables.
\end{corollary}

\section{Proof of Theorem 1}

\begin{lemma}
Suppose that $K_{h}$ is the sublevel set $u\leq h$ of a nonnegative solution
to
\[
\tilde{\sigma}_{2}(D^{2}u)=u_{11}(u_{22}+u_{33}+...+u_{nn})-u_{12}%
^{2}-...-u_{1n}^{2}=1.
\]
Then for all ellipsoids $E\subset K_{h}$ such that if $A:E\rightarrow B_{1}$
is affine diffeomorphism with
\[
A=Mx+\vec{b},
\]
we have
\[
\tilde{\sigma}_{2}(M^{2})\geq\frac{1}{4}\frac{1}{h^{2}}.
\]

\end{lemma}

\begin{proof}
Consider the function $v$ on $%
\mathbb{R}
^{n}$ defined by
\[
v(x)=h|A(x)|^{2}.
\]
On the boundary of $E$, we have
\[
v(x)=h\geq u.
\]
We have
\begin{align*}
Dv &  =2hM\left(  Mx+\vec{b}\right)  \\
D^{2}v &  =2hM^{2}.
\end{align*}
Thus
\[
\tilde{\sigma}_{2}(D^{2}v)=4h^{2}\tilde{\sigma}_{2}(M^{2}).
\]
Now suppose that
\[
\tilde{\sigma}_{2}(M^{2})<\frac{1}{4h^{2}}.
\]
Then
\[
\tilde{\sigma}_{2}(D^{2}v)<1,
\]
so $v$ is a supersolution to the equation, and must lie strictly above the
solution $u.$ \ But $v$ must vanish at $A^{-1}(0).$ \ Because $u$ is
nonnegative, this is a contradiction of the strong maximum principle. \ 
\end{proof}

\bigskip

\begin{proposition}
Suppose that $u$ is an entire convex solution to%
\[
\tilde{\sigma}_{2}(D^{2}u)=u_{11}(u_{22}+u_{33}+...+u_{nn})-u_{12}%
^{2}-...-u_{1n}^{2}=1
\]
Then
\[
\lim_{t\rightarrow\infty}u_{1}(t,0,..0)=\infty.
\]

\end{proposition}

\begin{proof}
Assume not. \ Instead assume that $u_{1}\leq A.$ \ \ Assume that $u(0)=0$ and
$Du(0)=0,$ adjusting $A$ if necessary. \ Then%
\[
u(t,0,...,0)=\int_{0}^{t}u_{1}(s)ds\leq\int_{0}^{t}Ads\leq At.
\]
Now consider the convex sublevel set $u\leq h.$ \ This must contain the point
\[
(\frac{h}{A},0,..,0).
\]
The level set $u=h$ intersect the other axes at
\begin{align*}
&  (0,a_{2}(h),0,...,)\\
&  (0,0,a_{3}(h),...,0)\\
&  \text{etc.}%
\end{align*}
This level set is convex. It must contain the simplex with the above points as
vertices, and this simplex must contain an ellipsoid $E$ which has an affine
transformation to the unit ball of the following form
\begin{align*}
A &  =Mx+\vec{b}\\
M &  =c_{n}\left(
\begin{array}
[c]{cccc}%
\frac{A}{h} &  &  & \\
& \frac{1}{a_{2}} &  & \\
&  & \frac{1}{a_{3}} & \\
&  &  & ...
\end{array}
\right)  .
\end{align*}
Thus
\[
M^{2}=c_{n}^{2}\left(
\begin{array}
[c]{cccc}%
\left(  \frac{A}{h}\right)  ^{2} &  &  & \\
& \left(  \frac{1}{a_{2}}\right)  ^{2} &  & \\
&  & \left(  \frac{1}{a_{3}^{2}}\right)  ^{2} & \\
&  &  & ...
\end{array}
\right)
\]
and
\[
\tilde{\sigma}_{2}(M^{2})=c_{n}^{2}\left(  \frac{A}{h}\right)  ^{2}\left(
\left(  \frac{1}{a_{2}}\right)  ^{2}+...+\left(  \frac{1}{a_{n}}\right)
^{2}\right)  \geq\frac{1}{4}\frac{1}{h^{2}}%
\]
with the latter inequality following from the previous lemma. \ 

Thus
\[
\left(  \frac{1}{a_{2}}\right)  ^{2}+\left(  \frac{1}{a_{3}}\right)
^{2}+...+\left(  \frac{1}{a_{n}}\right)  ^{2}\geq\frac{1}{c_{n}^{2}4A^{2}}.
\]
It follows that for some $i,$
\[
\frac{1}{a_{i}^{2}}\geq\frac{1}{4\left(  n-1\right)  c_{n}^{2}A^{2}}.
\]
That is
\[
a_{i}\leq2\sqrt{n-1}c_{n}A.
\]
Now to finish the argument, let%
\[
R=2\sqrt{n-1}c_{n}A.
\]
On a ball of radius $R$, there is some bound on the function (not a priori but
depending on $u$)\ say $\bar{U}.$ \ \ That is
\[
u(x)\leq\bar{U}\ \text{on }B_{R.}%
\]
Now by convexity for any large enough $h$ the level set $u=h$ is non-empty and
convex. \ Choose $h>\bar{U}$. \ \ According to the above argument, this level
set must intersect some axis at a point less than $R$ from the origin, which
is a contradiction. \ \

\end{proof}

Now using this Proposition, we may repeat the argument of He \cite[section
3]{HePJM}: \ Letting $z=u_{1}(t,x)$ the map
\begin{align*}
\Phi &  :%
\mathbb{R}
\times%
\mathbb{R}
^{n-1}\rightarrow%
\mathbb{R}
\times%
\mathbb{R}
^{n-1}\\
\Phi(t,x) &  =(z,x)
\end{align*}
is a diffeomorphism. \ Thus for $x$ fixed, there exists a unique $t=t(z,x)$
such that $z=u_{1}(t,x).$ \ Defining
\[
\theta(z,x)=t(z,x)
\]
the computations in \cite[section 3]{HePJM} yield that $\theta$ is a harmonic
function. \ It follows that $\frac{\partial\theta}{\partial z}=1/u_{11}$ is a
positive harmonic function, so must be constant. \ \ Now we have
\[
u(t,x)=at^{2}+tb(x)+g(x)
\]
which satisfies \cite[section 2]{HePJM}
\begin{align*}
\Delta b &  =0\\
\Delta g &  =\frac{1}{2a}\left(  1+\left\vert \nabla b\right\vert ^{2}\right)
.
\end{align*}
Letting $t=0$ we conclude that $g$ is convex. \ Letting $t\rightarrow\pm
\infty$ we conclude that $b$ is convex and concave, so must be linear. \ It
follows that $\left\vert \nabla b\right\vert $ is constant, and
\[
\Delta g-c\left\vert x\right\vert ^{2}%
\]
is a semi-convex harmonic function, which must be a quadratic. \ 

\bigskip

\section{Counterexamples}

We use the method described in \cite{warrensigmak} and restrict to $n=3.$
\ Consider
\[
u(t,x)=r^{2}e^{t}+h(t)
\]
where $r=\left(  x_{2}^{2}+x_{3}^{2}\right)  ^{1/2}.$ \ At any point we may
rotate $%
\mathbb{R}
^{2}$ so that $x_{2}=r$ and get
\[
D^{2}u=\left(
\begin{array}
[c]{ccc}%
r^{2}e^{t}+h^{\prime\prime}(t) & 2re^{t} & 0\\
2re^{t} & 2e^{t} & 0\\
0 & 0 & 2e^{t}%
\end{array}
\right)  .
\]
We compute
\[
\tilde{\sigma}_{2}\left(  D^{2}u\right)  =4e^{t}\left(  r^{2}e^{t}%
+h^{\prime\prime}(t)\right)  -4r^{2}e^{2t}=4e^{t}h^{\prime\prime}(t).
\]
Then
\[
u=r^{2}e^{t}+\frac{1}{4}e^{-t}%
\]
is a solution. \ 

Now defining complex variables
\begin{align*}
z_{1} &  =t+is\\
z_{2} &  =x+iy
\end{align*}
we can consider the function
\begin{equation}
u=\left(  x^{2}+y^{2}\right)  e^{t}+\frac{1}{4}e^{-t}.\label{cx}%
\end{equation}
The function \ref{cx} satisfies the equation complex Monge-Amp\`{e}re
equation
\[
\left(  \partial_{z_{1}}\partial_{z_{\bar{1}}}u\right)  \left(  \partial
_{z_{2}}\partial_{z_{\bar{2}}}u\right)  -\left(  \partial_{z_{1}}%
\partial_{z_{\bar{2}}}u\right)  \left(  \partial_{z_{2}}\partial_{z_{\bar{1}}%
}u\right)  =1.
\]
One can check that the induced Ricci-flat complex metric
\[
g_{i\bar{j}}=\partial_{z_{i}}\partial_{z_{\bar{j}}}u
\]
on $%
\mathbb{C}
^{2}$ is neither complete complete nor flat. \ 

\bibliographystyle{plain}
\bibliography{donaldson}

\end{document}